\documentclass[11pt]{article}
\usepackage{amsfonts}
\usepackage{amsmath}
\usepackage{amsthm}
\usepackage{amscd}
\usepackage{amsbsy}
\usepackage{graphicx}
\usepackage{indentfirst, latexsym, bm,amssymb}
\usepackage{bbding}

\usepackage[pagewise]{lineno}
\usepackage[colorlinks=true]{hyperref}
\advance\textwidth by +1.0in \advance\textheight by +1.0in
\advance\oddsidemargin by -0.5in \advance\evensidemargin by -1.0in
\advance\topmargin by -0.5in
\newtheorem {theorem} {Theorem}
\newtheorem {proposition} [theorem]{Proposition}
\newtheorem {corollary} [theorem]{Corollary}

\newtheorem {question} [theorem]{Question}

\newtheorem*{TheoremA}{Theorem A}

\usepackage{CJK}
\newtheorem{Defi}[theorem]{Definition}

\title{\large\bf On the Myrberg Limit Sets and Bowen-Margulis-Sullivan Measures for Visibility Manifolds without Conjugate Points}

\author{
\and Fei Liu
\thanks{College of Mathematics and Systems Science, Shandong University of Science and Technology, Qingdao, 266590, P.R. China.
e-mail: liufei$@$math.pku.edu.cn. }}
\date{\today}

\begin{document}
\maketitle

\begin{abstract}
   In this paper, we clarify the strong relationship between Myrberg type dynamics and the ergodic properties of the geodesic flows
on (not necessarily compact) uniform visibility manifolds without conjugate points.
We prove that the positivity of the Patterson-Sullivan measure of the Myrberg limit set is equivalent to the conservativity of the geodesic flow
with respect to the Bowen-Margulis-Sullivan measure.
Moreover we show that the Myrberg limit set is a full Patterson-Sullivan measure subset of the conical limit set.
These results extend the classical works of P. Tukia and B. Stratmann from hyperbolic manifolds to the manifolds without conjugate points.
\\


\end{abstract}


\section{\bf Introduction and Main Results}\label{intro}
\setcounter{section}{1}
\setcounter{equation}{0}\setcounter{theorem}{0}

This paper is a sequel to our previous paper~\cite{LLW},
with the aim of clarifying a fundamental aspect of the chaotic nature of the geodesic flows on the manifolds without conjugate points.

The theory of manifolds without conjugate points is one of the most challenging research topics in geometry and dynamical systems.
On the one hand, from the topological perspective, the absence of conjugate points results in strong restrictions on the topology of the manifolds.
However, the other side of the coin is, the fact that ``there are no conjugate points" provides (almost) no useful information on the local geometry of the manifold;
however, when studying the dynamics of geodesic flows, one needs to do a series of estimations using the local geometry.
Therefore, the dynamical aspects for the theory of manifolds without conjugate points is an important and highly challenging research topic.
For recent progresses on the various aspects of the dynamics of the manifolds without conjugate points, see the papers~\cite{CKP,CKW,GR,Ma,MR,PR,RR,Wu1,Wu3},
which contain many original ideas and utilize lots of new tools.

Let $(M,g)$ be a complete Riemannian manifold, and $(\widetilde{M},\widetilde{g})$ be its universal cover.
We use $d$ to denote the distance function induced by both $g$ on $M$ and $\widetilde{g}$ on $\widetilde{M}$ when there is no confusion.
Throughout this paper we always assume that the geodesics both in $M$ and $\widetilde{M}$ are of unit speed.
We denote by $\Gamma=\pi_1(M)$, i.e., $M=\Gamma\backslash\widetilde{M}$. $\Gamma$ can be viewed as a discrete subgroup of the isometry group of $\widetilde{M}$.

Let $T^1 M$ and $T^1\widetilde{M}$ denote the unit tangent bundles of $M$ and $\widetilde{M}$, respectively.
To simplify the notations, we use $\pi$ to denote the standard projection map from unit tangent bundle to the manifold, both for $T^1 M$ and $T^1\widetilde{M}$.
For any $v\in T^1 M$ or $v\in T^1 \widetilde{M}$, we use $c_{v}$ denote the unique unit speed geodesic with the initial conditions $c_{v}(0)=\pi(v)$ and $c'_{v}(0)=v$.
We use $$v\mapsto \phi_{t}(v)=c'_{v}(t)$$ to denote the $\mathbf{geodesic~flows}$ on the unit tangent bundles, both for $T^1 M$ and $T^1\widetilde{M}$.

One of the main results of our previous paper~\cite{LLW} is the following Hopf-Tsuji-Sullivan (HTS) dichotomy for visibility manifolds without conjugate points,
which says that the geodesic flow is either conservative or completely dissipative with respect to the Bowen-Margulis-Sullivan measure.
All the notions will be described in details in Section~\ref{geometry}. For more information about the HTS dichotomy,
refer to~\cite{LLW},~\cite{PPS} and \cite{Ro}, which contain further details.

\begin{theorem}[HTS Dichotomy for Visibility Manifolds~\cite{LLW}]\label{HTS}
    Let $M$ be a complete uniform visibility manifold without conjugate points that satisfies Axiom 2, and ${\{\mu_q\}}_{q\in\widetilde{M}}$
    be an $r$-dimensional Busemann density $(r\in \mathbb{R})$,
    $p\in\widetilde{M}$ is arbitrarily chosen and $\mathfrak{m}$ is the Bowen-Margulis-Sullivan (BMS) measure.
    Then either the geodesic flow $\phi_t$ is conservative or completely dissipative with respect to $\mathfrak{m}$ (namely, the HTS dichotomy holds). Moreover, the following statements are equivalent:
	\begin{enumerate}
		\item[1.] $\mu_p(L_c(\Gamma))=\mu_p(\widetilde{M}(\infty))$.
		\item[2.] The geodesic flow $\phi_t: T^{1}M \rightarrow T^{1}M $ is conservative with respect to the BMS measure $\mathfrak{m}$.
		\item[3.] The geodesic flow $\phi_t: T^{1}M \rightarrow T^{1}M $ is ergodic with respect to the BMS measure $\mathfrak{m}$.
		\item[4.] The $\Gamma$-action on $\widetilde{M}(\infty)\times\widetilde{M}(\infty)$ is ergodic with respect to $\mu_p\times\mu_p$.
		\item[5.] The Poincar\'e series $\sum_{\alpha\in\Gamma}\mathrm{e}^{-r \cdot d(p,\alpha p)}$ diverges.
	\end{enumerate}
\end{theorem}

The $L_{c}(\Gamma)$ appearing in the above Theorem is called the conical limit set, which plays a central role in the dynamics of the geodesic flows.

In this paper, we consider another important type of limit set, Myrberg limit set $L_{m}(\Gamma)$,
which was first introduced by Finnish mathematician Pekka Juhana Myrberg in~\cite{My}.
By definition, compared to the conical limit points, the Myrberg limit points appear more special.
In fact, it is well known that for the hyperbolic manifolds,
Myrberg limit set is a subset of conical limit set, and even a proper subset of conical limit set (cf. \cite{Tu}).
In Proposition~\ref{myr} of Section~\ref{geometry}, we show that for the visibility manifolds without conjugate points, this inclusion relationship still holds.
One of the key points of the proof is our previous estimations of the distance functions on the visibility manifolds without conjugate points.

One of the main results of this paper is the following result.

\begin{TheoremA}[Myrberg Type Dichotomy,~Theorem~\ref{Myr}]\label{ThmA}
Let $M$ be a complete uniform visibility manifold without conjugate points that satisfies Axiom 2,
${\{\mu_q\}}_{q\in\widetilde{M}}$ be a Patterson-Sullivan measure, $p\in\widetilde{M}$ is arbitrarily chosen and $\mathfrak{m}$ is the
corresponding $\delta_{\Gamma}$-dimensional BMS measure. Then the following assertions are equivalent:
\begin{enumerate}
		\item[1.] The Myrberg limit set $L_{m}(\Gamma)$ satisfies that $\mu_{p}(L_{m}(\Gamma)) >0$.
		\item[2.] The geodesic flow $\phi_t: T^{1}M \rightarrow T^{1}M $ is conservative with respect to the BMS measure $\mathfrak{m}$.
	\end{enumerate}
\end{TheoremA}

This theorem was first given by Pekka Tukia in his classical paper~\cite{Tu} for hyperbolic manifolds (see also~\cite{St}).
Based on a new observation for hyperbolic manifolds, Kurt Falk provided an elegant new proof of this conclusion in~\cite{Fa}.

Furthermore, we prove that when the geodesic flow $\phi_t: T^{1}M \rightarrow T^{1}M $ is conservative with respect to the BMS measure $\mathfrak{m}$,
the Myrberg limit set $L_{m}(\Gamma)$ has full $\mu_{p}$-measure in the conical limit set $L_{c}(\Gamma)$ (Corollary~\ref{cor}).
As a concequence,
the first item of Theorem A can be replaced by ``The Myrberg limit set $L_{m}(\Gamma)$ has full $\mu_{p}$-measure".
In Section~\ref{BMS}, we also provide a characterization of the non-wandering set of the geodesic flow and prove the uniqueness of the BMS measures.

This paper is organized as follows. In Section~\ref{geometry} we introduce the notions and results which are required in the sequel,
and provide some interesting conclusions in geometry and the theory of dynamical systems.
In Section~\ref{BMS} we discuss the Myrberg limit set and Bowen-Margulis-Sullivan measures on visibility manifolds in detail, and prove the main conclusions of this paper.
In Section~\ref{question}, we list some related questions for further study.

\section{\bf{Geometry of Visibility Manifolds}}\label{geometry}
In this section, we summarize the notions that we needed, and we also investigate some geometric properties of visibility manifolds.

Let $c$ be a geodesic in $M$. We say two points $p_1=c(t_1)$ and $p_2=c(t_2)$ on the geodesic $c$ are $\mathbf{conjugate}$,
if there exists a non-trivial Jacobi field $\bm{J}$, such that
$$\bm{J}(t_1)=0=\bm{J}(t_2).$$

$M$ is called a $\mathbf{manifold~without~conjugate~points}$ if no geodesic on $M$ admits conjugate points.
By definition, it's easy to see that non-positively curved manifolds have no conjugate points,
and meanwhile the standard two dimensional sphere $S^{2}$ has conjugate points since the antipodal points are conjugate to each other.
But the absence of conjugate points does not necessarily imply the manifold is non-positively curved.
In fact, there are examples of manifolds without conjugate points that admit some regions with positive sectional curvature (cf.~\cite{Gu}).

In a sense, the condition of ``without conjugate points" is so broad that we need to add some additional conditions to derive important geometric properties
and fine (globally and locally) estimations. Visibility is such a condition, which was first introduced by Patrick Eberlein.

$\widetilde{M}$ is called a $\mathbf{visibility~manifold}$ if for any $p\in\widetilde{M}$ and $\epsilon>0$, there exists a
constant $R_{p,\epsilon}>0$, such that for any geodesic (segment) $c:[a,b] \to \widetilde{M}$, $d(p,c)\geq R_{p,\epsilon}$ implies that $\measuredangle_p(c(a),c(b))\leq\epsilon$. Here we allow $a$ and $b$ to be infinity.
If the constant $R$ is independent of the choice of the point $p$,
$\widetilde{M}$ is called a $\mathbf{uniform~visibility~manifold}$. $M$ is called a
$\mathbf{(uniform)~visibility~manifold}$ if its universal cover $\widetilde{M}$ does.

Visibility is a very natural condition. It is proved in \cite{BO} that the manifolds with negative upper-bound sectional curvature are uniform visibility manifolds.
On the other hand, there are many visibility manifolds admit some regions with zero and positive sectional curvature. For example,
a closed surface without conjugate points and genus greater than $1$ is a uniform visibility manifold.
For more information, see \cite{Eb1}, which contains almost all the useful geometric properties about visibility manifolds without conjugate points.

By Cartan-Hadamard theorem we know that $\widetilde{M}$ is diffeomorphic to $\mathbb{R}^{n}$ where $n=\mathrm{dim} M$.
We will add a boundary to $\widetilde{M}$, and make it a compact space under the so called cone topology.

Two geodesics $c_{1}$ and $c_{2}$ in $\widetilde{M}$ are called $\mathbf{positively~asymptotic}$,
if there is a positive constant $C$ such that
$$d(c_1(t),c_2(t))\leq C, \quad\forall t\geq 0.$$
The positively asymptotic is an equivalence relation among geodesics on $\widetilde{M}$,
the set of the equivalence classes is called the $\mathbf{ideal~boundary}$ and is denoted by $\bf{\widetilde{M}(\infty)}$.

If $\widetilde{M}$ is a visibility manifold without conjugate points, Eberlein (cf. \cite{Eb1}) showed that for any point $p\in\widetilde{M}$ and $\xi\in\widetilde{M}(\infty)$, there exists a unique $v\in T^{1}_{p}\widetilde{M}$ such that
$c_{v}(+\infty)=\xi$, where $c_{v}$ is the unique geodesic satisfies $c(0)=p$ and $c'(0)=v$.
Therefore $\widetilde{M}(\infty)$ is homeomorphic to $T^{1}_{p}\widetilde{M}$,
which is homeomorphic to the $(n-1)$-dimensional unit sphere $\mathbb{S}^{n-1}$.

Given two different points $x, y\in\mathbf{\overline{\widetilde{M}}}\triangleq \widetilde{M}\cup\widetilde{M}(\infty)$
and let $\mathbf{c_{x,y}}$ be the unique geodesic connecting $x$ and $y$, if in addition $x\in \widetilde{M}$, we parametrize the geodesic $c_{x,y}$
by $c_{x,y}(0)=x$. We list the following notations of the angles between geodesics:
\begin{displaymath}
	\begin{aligned}
		\measuredangle_p(x,y)      & =\measuredangle(c'_{p,x}(0),c'_{p,y}(0)),\quad x,y\neq p;                        \\
		\measuredangle_p(x,\bm{v}) & =\measuredangle(c'_{p,x}(0),\bm{v}),\quad \bm{v}\in T_p^1\widetilde{M}, ~x\neq p; \\
        C(\bm{v},\epsilon)    & =\{x\in \overline{\widetilde{M}} - \{p\}\mid \measuredangle_p(x,\bm{v})<\epsilon\}, ~~\bm{v}\in T_p^1\widetilde{M};                                                                     \\
		C_\epsilon(\bm{v})    & =\{c_{\bm{w}}(+\infty)\mid \bm{w}\in T^1_p\widetilde{M},\measuredangle(\bm{v},\bm{w})<\epsilon\}\subset\widetilde{M}(\infty), ~~\bm{v}\in T_p^1\widetilde{M};\\
		TC(\bm{v},\epsilon,r) & =\{x\in\widetilde{M}\cup\widetilde{M}(\infty)\mid \measuredangle_p(x,\bm{v})<\epsilon, d(p,x)> r\},
~~\bm{v}\in T_p^1\widetilde{M};\\
	\end{aligned}
\end{displaymath}

The last one $TC(\bm{v},\epsilon,r)$ is called the $\mathbf{truncated~cone~with~axis~\bm{v}~and~angle~\epsilon}$.
For any point $\xi\in\widetilde{M}(\infty)$, the set of truncated cones containing this point actually forms local bases and hence, forms the bases for a topology $\tau$. This topology is unique and usually called the $\mathbf{cone~topology}$. Under this topology,
$\overline{\widetilde{M}}$ is homeomorphic to the n-dimensional unit closed ball in $\mathbb{R}^{n}$.
More precisely, for any $p\in \widetilde{M}$, let $B_{p}$ be the closed unit ball in the tangent space $T_{p}\widetilde{M}$, i.e.,
$B_{p} \triangleq \{ v \in T_{p}\widetilde{M} \mid \parallel v\parallel \leqslant 1\}$,
then if $M$ is a visibility manifold without conjugate points, the following map
\[
h:B_{p} \rightarrow \overline{\widetilde{M}},~~~ v \mapsto h(v) =
\left\{
\begin{aligned}
\exp_{p}(\frac{v}{1-\parallel v\parallel}) & \quad \text{if }  \parallel v \parallel < 1, \\
c_{v}(+\infty)~~~~~~~~~ & \quad \text{if } \parallel v \parallel = 1,
\end{aligned}
\right.
\]
is a homeomorphism under the cone topology.
For more details, refer to~\cite{Eb1,EO}.

Denote by $L(\Gamma)\triangleq\widetilde{M}(\infty)\cap\overline{\Gamma(p)}$,
where $\overline{\Gamma(p)}$ is the closure of the orbit of the $\Gamma$-ation at $p$ under the cone topology.
$L(\Gamma)$ is called the $\mathbf{limit~set}$ of $\Gamma$.
Due to the visibility axiom, one can check that $L(\Gamma)$ is independent of the choice of the point $p$.
$\Gamma$ is called $\mathbf{non}$-$\mathbf{elementary}$ if $\#L(\Gamma)=\infty$.

It's easy to see that the limit set $L(\Gamma)$ is a $\Gamma$-invariant closed subset of the ideal boundary.
In fact it is precisely the set of points in $\overline{\widetilde{M}}$ where the proper discontinuity fails (cf. \cite{EO}).

Two (not necessarily distinct) points $\xi, \eta \in \widetilde{M}(\infty)$ are called $\mathbf{\Gamma}$-$\mathbf{dual}$,
if there exists a sequence $\{\alpha_{n}\}^{\infty}_{n=1}\subset \Gamma$ and a point $p \in \widetilde{M}$
(hence for all points $p \in \widetilde{M}$ due to visibility) such that under the cone topology,
$$\alpha^{-1}_{n}(p)\rightarrow \xi, ~~~\alpha_{n}(p)\rightarrow \eta.$$

For the visibility manifolds without conjugate points, we have the following properties.

\begin{proposition}\label{prop_2_7}
	Let $\widetilde{M}$ be a simply connected uniform visibility manifold without conjugate points.
	\begin{enumerate}
   \item\cite{Eb1} The following map is continuous:
		\begin{displaymath}
			\begin{aligned}
				\Psi:~& T^1\widetilde{M}\times[-\infty,\infty] & \to     &~\widetilde{M}\cup\widetilde{M}(\infty), \\
				       & \qquad(\bm{v},t)                       & \mapsto &~c_{\bm{v}}(t).
			\end{aligned}
		\end{displaymath}
    \item\cite{Eb1} For any two points $\xi\neq\eta$ on the ideal boundary $\widetilde{M}(\infty)$,
        there exists at least one connecting geodesic from $\xi$ to $\eta$.
    \item\cite{Eb1} If $\Gamma$ is non-elementary, any two points in $L(\Gamma)$ are $\Gamma$-dual.
	\item\cite{LW} For any two positively asymptotic geodesics $c_1$ and $c_2$,
		\begin{displaymath}
		    d(c_1(t),c_2(t))\leq 2R_{\frac{\pi}{2}}+3d(c_1(0),c_2(0)),~~~t>0.	
		\end{displaymath}
        Here $R_{\frac{\pi}{2}}$ is the uniform visibility constant.
    \item\cite{LWW} For any $\bm{v}\in T^1{\widetilde{M}}$ and positive constants $R,\epsilon$,
        there is a constant $L=L(\epsilon, R)$, such that for any $t>L$,
        \begin{displaymath}
            B(c_{\bm{v}}(t),R)\subset C(\bm{v},\epsilon).
        \end{displaymath}
        Here $B(c_{\bm{v}}(t),R)$ is the open ball centered at $c_{\bm{v}}(t)$ with radius $R$.
	\end{enumerate}
\end{proposition}

Proposition~\ref{prop_2_7} (5) is proved in \cite{LWW} for rank 1 manifolds without focal points.
One can check that it follows from the uniform visibility condition,
and the constant $L$ here does not depend on the choice of $\bm{v}$.

A limit point $\xi\in\widetilde{M}(\infty)$ is called a $\mathbf{conical~limit~point}$,
if for any point $p\in\widetilde{M}$, there exists a constant $C=C(p)>0$ and $\{\alpha_n\}^{\infty}_{n=1}\subset\Gamma$,
such that $d(\alpha_n p, c_{p,\xi})\leq C$ and $\alpha_n p\to\xi$.
Denote the set of all conical limit points by $\mathbf{L_c(\Gamma)}$.

A limit point $\xi \in \widetilde{M}(\infty)$ is called a $\mathbf{Myrberg~limit~point}$
if for any pair of points $\eta, \eta' \in L(\Gamma)$ with $\eta \neq\eta'$, and for any point $x \in \widetilde{M}$,
there exists a sequence $\{\alpha_{n}\}^{\infty}_{n=1}\subset \Gamma$ such that
$$\alpha_{n}x\rightarrow \eta, ~~~\alpha_{n}\xi\rightarrow \eta'.$$
The set of all Myrberg limit points is denoted by $\mathbf{L_{m}(\Gamma)}$.

When the manifold is negatively curved, it is well known that $L_{m}(\Gamma) \subset L_c(\Gamma)$.
We extend this result to the visibility manifolds.

\begin{proposition}\label{myr}
Let $\widetilde{M}$ be a simply connected uniform visibility manifolds without conjugate points,
then $L_{m}(\Gamma) \subset L_c(\Gamma)$.
\end{proposition}
\begin{proof}
Let $\xi$ be a Myrberg limit point, by definition, for any limit points $\eta \neq \eta' \in L(\Gamma)$,
and any $x \in \widetilde{M}$,
there exists a sequence $\{\alpha_{n}\}^{\infty}_{n=1}\subset \Gamma$ such that
$\alpha_{n}x\rightarrow \eta, ~\alpha_{n}\xi\rightarrow \eta'$. By Proposition~\ref{prop_2_7} (1),
we know that
$$\measuredangle_x(\alpha_{n}x,\alpha_{n}\xi)\rightarrow \measuredangle_x(\eta,\eta')>0.$$
Thus there exists an $\epsilon_{0}>0$ such that
\begin{equation}\label{eq2.1}
		 \measuredangle_x(\alpha_{n}x,\alpha_{n}\xi) \geqslant \epsilon_{0}>0,~~~n \in \mathbb{N}^{+}.
\end{equation}
Since $\widetilde{M}$ satisfies the axiom of uniform visibility, by inequality~\eqref{eq2.1} we know that there is a constant $R\triangleq R_{\epsilon_{0}}>0$ satisfies
$$d(x, c_{n})\leqslant R, ~~~n \in \mathbb{N}^{+},$$
where $c_{n}  \triangleq c_{\alpha_{n}x, \alpha_{n}\xi}=\alpha_{n}\circ c_{x,\xi}$
is the geodesic ray that from $\alpha_{n}x$ to $\alpha_{n}\xi$.
Therefore we can choose a point $p_{n}\in c_{n}$ with
\begin{equation}\label{eq2.2}
		d(x, p_{n})\leqslant R,  ~~~n \in \mathbb{N}^{+}.
\end{equation}
Let $\bm{v}_{n}$ be the tangent vector of geodesic $c_{n}$ at $p_{n}$, i.e.,
$\bm{v}_{n}\in T_{p_{n}}^{1}\widetilde{M}$ with $c_{\bm{v}_{n}}(+\infty)=\alpha_{n}\xi$,
thus we get a sequence of vectors $\{\bm{v}_{n}\}^{\infty}_{n=1}$.
By inequality~\eqref{eq2.2}, passing to a sub-sequence if needed,
we can assume that
\begin{equation}\label{eq2.3}
		\lim_{n\rightarrow +\infty}\bm{v}_{n} = \bm{v} \in T_{p}^{1}\widetilde{M}.
\end{equation}

By proposition~\ref{prop_2_7} (1), $p$ is at the connecting geodesic $c_{\eta,\eta'}$ and $\bm{v}$ is tangent to this geodesic at $p$.
Without loss of generality, by~\eqref{eq2.3}, we can assume that $$d(p_{n},p)<1, ~~~n\in \mathbb{N}^{+}.$$
Thus
$$
d(\alpha^{-1}_{n}p, c_{x,\xi})  =  d(p, \alpha_{n}\circ c_{x,\xi})
			                    =  d(p, c_{n})
			                    \leqslant  d(p,p_{n}) < 1.
$$
Therefore we have
\begin{displaymath}
		\begin{aligned}
			d(\alpha^{-1}_{n}x, c_{x,\xi}) & = d(x, \alpha_{n}\circ c_{x,\xi})                   \\
			                               & = d(x, c_{n}) \\
			                               & \leqslant d(x,p)+d(p,c_{n})                                               \\
			                               & < (R+1)+1                                    \\
                                           & = R+2, ~~~n\in \mathbb{N}^{+}.                                    \\
		\end{aligned}
\end{displaymath}
This shows that $\xi \in L_c(\Gamma)$.

\begin{figure}[hb]
		\centering
		\includegraphics[width=0.9\textwidth]{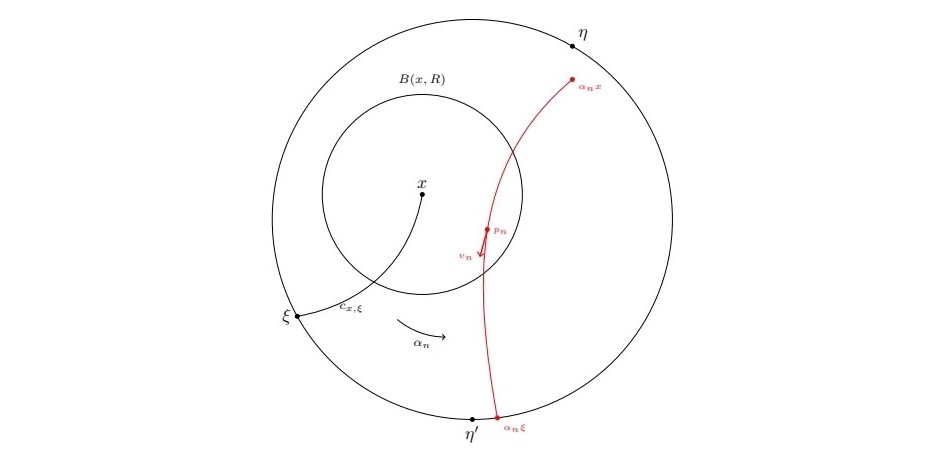}
	\end{figure}

\end{proof}

The $\mathbf{Busemann~function}$ is defined as follows:
	\begin{displaymath}
		\begin{aligned}
			\beta:~&~\widetilde{M}(\infty)\times\widetilde{M}\times\widetilde{M} & \to     &~\mathbb{R}                                               , \\
			        & \qquad(\xi,p,x)                                              & \mapsto &~\beta_{\xi}(p,x)\triangleq\lim_{t\to\infty}\{d(p,c_{x,\xi}(t))-t\}.
		\end{aligned}
	\end{displaymath}
One can check that $|\beta_{\xi}(p,x)|\leq d(p,x)$ by the triangle inequality.

The level sets of a Busemann function are called $\mathbf{horospheres}$. More precisely, the set
$$ H_{\xi}(p)\triangleq\{x\in\widetilde{M}\mid\beta_{\xi}(p,x)=0\}$$
is called the horosphere with center $\xi\in\widetilde{M}(\infty)$, based at $p\in\widetilde{M}$.

Let $\mathbf{\widetilde{M}^2(\infty)}\triangleq\widetilde{M}(\infty)\times\widetilde{M}(\infty)-\{(\xi,\xi)\mid \xi\in\widetilde{M}(\infty)\}$.
Fix a point $p\in\widetilde{M}$, the $\mathbf{Gromov~product~at~p}$ is given by
\begin{displaymath}
	\begin{aligned}
		\beta_p:~& \widetilde{M}^2(\infty) & \to     &~\mathbb{R}^{+},                     \\
		          & (\xi,\eta)              & \mapsto &~\beta_p(\xi,\eta)\triangleq\beta_{\xi}(p,x)+\beta_{\eta}(p,x). \\
	\end{aligned}
\end{displaymath}
Here $x\in\widetilde{M}$ is any point on the connecting geodesic $c_{\xi,\eta}$.

One can see that the Gromov product doesn't depend on the choice of the connecting geodesic $c_{\xi,\eta}$, nor on the choice of $x$. Geometrically, $\beta_p(\xi,\eta)$ is the length of $c_{\xi,\eta}$ bounded between the horospheres $H_{\xi}(p)$ and $H_{\eta}(p)$.

\begin{Defi}\label{def2}
Let $r$ be a positive real number, a family of finite Borel measure ${\{\mu_p\}}_{p\in\widetilde{M}}$ on the ideal boundary
is called an $\mathbf{r}$-$\mathbf{dimensional~Busemann~density}$, if
	\begin{enumerate}
		\item For each $p\in\widetilde{M}$, the support of $\mu_p$ is contained in $L(\Gamma)$.
		\item For any $p, q\in\widetilde{M}$ and $\mu_p$-a.e. $\xi\in\widetilde{M}(\infty)$ on the ideal boundary, we have
              \begin{equation}\label{eq3.1}
		           \frac{\mathrm{d}\mu_p}{\mathrm{d}\mu_q}(\xi)=\mathrm{e}^{-r\beta_{\xi}(p,q)},
              \end{equation}
		      where $\beta_{\xi}(p,q)$ is the Busemann functions defined above.
		\item ${\{\mu_p\}}_{p\in\widetilde{M}}$ is $\Gamma$-equivariant, i.e., for any Borel subset $A\subset\widetilde{M}(\infty)$ and $\alpha\in\Gamma$,
		       \begin{equation}\label{eq3.2}
		           \mu_{\alpha p}(\alpha A)=\mu_p(A).
              \end{equation}
	\end{enumerate}
\end{Defi}

As is well known that (see Proposition~\ref{prop_2_7}(2)), visibility insures that for any two points $\xi\neq\eta$ on the ideal boundary,
there exists at least one connecting geodesic, but this doesn't mean that the connecting geodesic in unique!
In fact, when the manifold is non-positively curved or has no focal points,
the non-uniqueness of the connecting geodesics implies the existence of a flat strip.
On the other hand, Keith Burns construct a clever example showed that the flat strip theorem is not valid for manifolds without conjugate points (cf. \cite{Bu}).

The failure of the flat strip theorem caused trouble in constructing the Bowen-Margulis-Sullivan measure, which is the theme of this notes. We need the following concept, first proposed by Eberlein and O'Neill for manifolds of non-positive curvature.

\begin{Defi} [cf. \cite{EO}] \label{def1}
Let $\widetilde{M}$ be a complete simply connected Riemannian manifold without conjugate points,
we call it satisfies $\mathbf{Axiom~2}$, if for any points $\xi,\eta \in \widetilde{M}(\infty)$ with $\xi\neq\eta$, there exists at most one geodesic connecting them.
\end{Defi}

Let ${\{\mu_q\}}_{q\in\widetilde{M}}$ be an $r$-dimensional Busemann density.
Fix a point $p\in\widetilde{M}$, we can define a measure on $\widetilde{M}^2(\infty)$ by
\begin{displaymath}
	\mathrm{d}\widetilde{\mu}(\xi,\eta)=\mathrm{e}^{r\cdot\beta_p(\xi,\eta)}\mathrm{d}\mu_p(\xi)\mathrm{d}\mu_p(\eta),
\end{displaymath}

It's easy to check that this definition does not depend on the choice of the point $p$, and it is $\Gamma$-invariant by the equivariance of the Patterson-Sullivan measure.
Furthermore, if $\widetilde{M}$ is a uniform visibility manifold without conjugate points that satisfies Axiom $2$, this measure induces a $\phi_t$- and $\Gamma$-invariant measure $\widetilde{\mathfrak{m}}$ on $T^1\widetilde{M}$ by
\begin{displaymath}
    \widetilde{\mathfrak{m}}(A)=\int_{\widetilde{M}^2(\infty)}\text{Length}(c_{\xi,\eta}\cap \pi(A))\,\mathrm{d}\widetilde{\mu}(\xi,\eta)
\end{displaymath}
for any Borel set $A\subset T^1\widetilde{M}$. Here $c_{\xi,\eta}$ is the unique connecting geodesic from $\xi$ to $\eta$
since $\widetilde{M}$ satisfies the Axiom of uniform visibility and Axiom $2$,
and $\pi:T^1 \widetilde{M}\to\widetilde{M}$ is the standard projection.

By the $\Gamma$-invariance of $\widetilde{\mathfrak{m}}$ on $T^1\widetilde{M}$,
the measure induces a $\phi_t$-invariant measure $\mathfrak{m}$ on $T^1M$ by the standard projection map,
which is known as the $\mathbf{r}$-$\mathbf{dimensional}$~$\mathbf{Bowen}$-$\mathbf{Margulis}$-$\mathbf{Sullivan~(BMS)~measure}$.
The BMS measure was first constructed on hyperbolic manifolds by Dennis Sullivan in~\cite{Su1},
then he proved that it is just the unique measure of maximal entropy for compact hyperbolic manifolds (cf. \cite{Su2}).

$\bm{v}\in T^1M$ is called a $\mathbf{conservative~point}$ of the geodesic flow, if there exists a compact subset $A \subset T^1M$
and a sequence of real numbers $\{t_{n}\}^{\infty}_{n=1}$, such that $t_n\to+\infty$ and $\phi_{t_n}(\bm{v})\in A$ for all $n$.
$\bm{v}\in T^1M$ is called a $\mathbf{dissipative~point}$, if for any compact subset $A\subset T^1M$, there is a real number $t_A > 0$, such that $\phi_t(\bm{v})\notin A$ for all $t>t_A$.

Let $M_C$ and $M_D$ be the sets of all conservative points and the set of all dissipative points, respectively.
For each $\bm{v}\in T^1M$, it is either a conservative points or a dissipative point, thus we have $T^1M=M_C\sqcup M_D$.
The geodesic flow is called $\mathbf{conservative}$ with respective to $\mathfrak{m}$ if $\mathfrak{m}(M_D)=0$.
Similarly, it is called $\mathbf{completely~dissipative}$ with respective to $\mathfrak{m}$ if $\mathfrak{m}(M_C)=0$.

An isometry $\alpha \in \mathrm{Iso}(\widetilde{M})$ is called $\mathbf{axial}$
if there is a constant $T>0$ and a geodesic $c$ such that
$$ \alpha \circ c(t)=c(t+T),~~~t\in \mathbb{R}.$$
The geodesic $c$ is called an $\mathbf{axis}$ of the axial element $\alpha$.

\begin{proposition}\label{axial}
Let $\widetilde{M}$ be a simply connected uniform visibility manifolds without conjugate points,
and $\alpha \in Iso(\widetilde{M})$ be an axial element, then
\begin{enumerate}
   \item $\alpha$ generates a discrete group.
   \item If geodesic $c$ is an axis of $\alpha$, then for any $p \in \widetilde{M}$, we have
   $$\alpha^{-n}(p)\rightarrow c(-\infty),~~\alpha^{n}(p)\rightarrow c(+\infty),~~~~n\rightarrow +\infty.$$
   \item All axes of $\alpha$ are equivalent. i.e., they have the same endpoints in the ideal boundary.
   \item If there exists $\beta \in Iso(\widetilde{M})$ and $n\in \mathbb{N}^{+}$ such that
   $\beta \circ \alpha^{n} = \alpha^{n} \circ \beta$, then $\beta$ fix the two endpoints of the axes of $\alpha$.
	\end{enumerate}
\end{proposition}
\begin{proof}
(1). Let $c$ be an axis of the axial element $\alpha$, i.e., there exists a positive $T>0$,
such that $\alpha\circ c(t)=c(t+T), ~t\in \mathbb{R}$. Thus for any $p\in \widetilde{M}$ and $n \in \mathbb{Z}$,
\begin{displaymath}
		\begin{aligned}
			d(p, \alpha^{n}(p)) & \geqslant d(c(0), \alpha^{n}(c(0))) - d(p,c(0))-d(\alpha^{n}(p),\alpha^{n}c(0))     \\
			                               & = d(c(0), c(nT)) - 2d(p,c(0)) \\
			                               & = |n| T -2d(p,c(0)),                                               \\
		\end{aligned}
\end{displaymath}
thus $\alpha$ generates a discrete group.

(2). Let $c$ be an axis of $\alpha$ as in (1), i.e., $\alpha\circ c(t)=c(t+T), ~t\in \mathbb{R}$.
Denote by $q\triangleq c(0)$ and $\theta_{n}\triangleq \measuredangle_{q}(c(+\infty),\alpha^{n}(p)), ~n\in \mathbb{N}^{+}$.
Since $d(\alpha^{n}(q), \alpha^{n}(p))=d(q,p)$ is independent of $n$ and $d(q,c(nT))=nT\rightarrow \infty$,
we know that
$$d(q,c_{n})\rightarrow +\infty, ~~~n\rightarrow +\infty,$$
where $c_{n}=c_{\alpha^{n}(p),\alpha^{n}(q)}$ is the connecting geodesic segment from $\alpha^{n}(p)$ to $\alpha^{n}(q)$.
Thus by the Axiom of uniform visibility, we have $\theta_{n}\rightarrow 0$.
Therefore by the definition of the cone topology, we have $\alpha^{n}(p)\rightarrow c(+\infty),~n\rightarrow +\infty.$
Similar argument will show that $\alpha^{-n}(p)\rightarrow c(-\infty),~n\rightarrow +\infty.$

(3). Let $c_{1}$ and $c_{2}$ be two different axes of the axial element $\alpha$ with
$$\alpha \circ c_{1}(t)=c_{1}(t+T_{1}), ~~~t\in \mathbb{R},$$
$$\alpha \circ c_{2}(t)=c_{2}(t+T_{2}), ~~~t\in \mathbb{R},$$
where $T_{1}>0$ and $T_{2}>0$ are real constants.
First we'll show that $T_{1}=T_{2}$.
By definition, for any $n\in \mathbb{N}^{+}$, we have
\begin{displaymath}
		\begin{aligned}
			nT_{2} & = d(c_{2}(0), \alpha^{n}\circ c_{2}(0)))     \\
			       & \geqslant d(c_{1}(0), \alpha^{n}\circ c_{1}(0))) - d(c_{1}(0),c_{2}(0))-d(\alpha^{n}\circ c_{1}(0),\alpha^{n}\circ c_{2}(0)) \\
			       & = nT_{1} -2d(c_{1}(0),c_{2}(0)).                                              \\
		\end{aligned}
\end{displaymath}
Thus we have
\begin{equation}\label{eq2.4}
		T_{2} \geqslant T_{1} - \frac{2}{n}d(c_{1}(0),c_{2}(0)), ~~~n\in \mathbb{N}^{+}.
\end{equation}
Similarly
\begin{equation}\label{eq2.5}
		T_{1} \geqslant T_{2} - \frac{2}{n}d(c_{1}(0),c_{2}(0)), ~~~n\in \mathbb{N}^{+}.
\end{equation}
By~\eqref{eq2.4} and~\eqref{eq2.5}, we have $T_{1}=T_{2}$. Moreover, we have
$$d(c_{1}(t),c_{2}(t))\leqslant d(c_{1}(0),c_{2}(0)) + 2T_{1},~~~t\in \mathbb{R}.$$
Thus $c_{1}$ and $c_{2}$ are equivalent.

(4). By (3) we know if the axial element $\alpha$ have more than one axis, all of them are equivalent geodesics,
thus have the same endpoints in the ideal boundary $\widetilde{M}(\infty)$.
Let $\xi$ and $\eta$ be the endpoints of these axes, choose one axis $c$ with $c(-\infty)=\xi$ and $c(+\infty)=\eta$.
Since $\beta \circ \alpha^{n} = \alpha^{n} \circ \beta$, we have
$\beta \circ \alpha^{n} (c(\pm\infty))= \alpha^{n} \circ \beta(c(\pm\infty))$,
thus we get $\beta(c(\pm\infty)) = \alpha^{n} \circ \beta(c(\pm\infty))$, so $\beta(c(\pm\infty))$ are the fixed points in the ideal boundary of $\alpha^{n}$.
While the set of fixed points of $\alpha^{n}$ are just $\{c(+\infty), c(-\infty)\}$ (cf. \cite{Eb1} Proposition 2.6(2)),
therefore there are only two cases: either
\begin{equation}\label{eq2.6}
		\beta(c(+\infty))=c(+\infty),~~~\beta(c(-\infty))=c(-\infty);
\end{equation}
or
\begin{equation}\label{eq2.7}
		\beta(c(+\infty))=c(-\infty),~~~\beta(c(-\infty))=c(+\infty).
\end{equation}

Now we'll show that~\eqref{eq2.7} is not true, thus~\eqref{eq2.6} must be true, thus the conclusion is valid.
In fact, $\beta \circ \alpha^{n}(c(0)) = \alpha^{n} \circ \beta(c(0))$, i.e.,
 $\beta (c(nT)) = \alpha^{n} \circ \beta(c(0))$, by (3) of this proposition we know that $\beta(c(+\infty))=c(+\infty)$,
 thus~\eqref{eq2.6} is valid.
\end{proof}

\begin{proposition}\label{axial2}
Let $\widetilde{M}$ be a simply connected uniform visibility manifolds without conjugate points,
$\Gamma$ be a discrete subgroup of the isometry group $Iso(\widetilde{M})$, and $\alpha,~\beta$ be two elements in $\Gamma$.
Suppose that $\alpha$ is an axial element, and geodesic $c$ is an axis of $\alpha$, if $c(+\infty)$ is a fixed point of $\beta$,
then there exists a non-zero integer $n$ such that $\beta \circ \alpha^{n} = \alpha^{n} \circ \beta$,
and $c(-\infty)$ is another fixed point of $\beta$.
\end{proposition}
\begin{proof}
Since $\beta(c(+\infty))=c(+\infty)$, we know that the geodesics $c$ and $\beta(c)$ are positively asymptotic,
thus there exists a constant $C>0$ such that
$$d(c(t),(\beta\circ c)(t))\leqslant C,~~~t\geqslant 0. $$
Suppose that there is a $T>0$ with $$\alpha(c(t))=c(t+T), ~~~t\in \mathbb{R},$$
then we have
\begin{displaymath}
		\begin{aligned}
			d(c(0),~\alpha^{-n}\circ \beta \circ \alpha^{n}(c(0)))
        & = d(\alpha^{n} (c(0)),~\beta \circ \alpha^{n}(c(0)))     \\
	    & = d(c(nT), ~\beta\circ c(nT))                             \\
	    & \leqslant C, ~~~~~~n\in \mathbb{N}^{+}.                    \\
		\end{aligned}
\end{displaymath}
By the fact that $\Gamma$ is a discrete group, we know that there exists a constant $N>0$,
such that for any $n,m \geqslant N$, we have
$$\alpha^{-n}\circ \beta \circ \alpha^{n} = \alpha^{-m}\circ \beta \circ \alpha^{m},$$
therefore $\alpha^{m-n}\circ \beta = \beta \circ \alpha^{m-n}$. Furthermore,
by Proposition~\ref{axial}(4) $\beta$ also fixes $c(-\infty)$.
\end{proof}

\begin{theorem}\label{mini}
Let $\widetilde{M}$ be a simply connected uniform visibility manifolds without conjugate points,
and $\Gamma \subset Iso(\widetilde{M})$ be a discrete group, then
the $\Gamma$-action on the limit set is minimal, i.e., $\forall \xi \in L(\Gamma)$, $\overline{\Gamma\xi}=L(\Gamma)$.
\end{theorem}
\begin{proof}
For any $\xi \in L(\Gamma)$, according to Proposition~\ref{prop_2_7}(3), for each $\eta\in L(\Gamma)$, $\xi$ and $\eta$ are $\Gamma$-dual.
Thus by Proposition 2.5 in \cite{Eb1},
there are two sequences of open neighbourhoods $\{U_{n}\}^{+\infty}_{n=1}$ and $\{V_{n}\}^{+\infty}_{n=1}$ in $\widetilde{M}(\infty)$,
$\{\alpha_{n}\}^{+\infty}_{n=1}\subset \Gamma$, and $\zeta \in \widetilde{M}(\infty)$,
such that
$$U_{n+1} \subset U_{n},~~~V_{n+1} \subset V_{n},~~n\in \mathbb{N};    $$
$$\bigcap^{+\infty}_{n=1}U_{n}=\{\xi\},~~~\bigcap^{+\infty}_{n=1}V_{n}=\{\eta\};    $$
$$\alpha^{-1}_{n}\zeta\in U_{n},~~~\alpha_{n}\zeta\in V_{n}.$$
Therefore $$\alpha^{-1}_{n}\zeta\rightarrow \xi,~~~\alpha_{n}\zeta\rightarrow \eta.$$
Finally we get $$\lim_{n\rightarrow+\infty}\alpha^{2}_{n}\xi = \lim_{n\rightarrow+\infty} \alpha_{n}\zeta = \eta.$$
\end{proof}


\section{\bf Bowen-Margulis-Sullivan Measures on (not Necessarily Compact) Visibility Manifolds}\label{BMS}
\setcounter{equation}{0}\setcounter{theorem}{0}
In this section, we'll prove the main results of this paper.

There is a classical method, due to Patterson (cf. \cite{Pa}), to construct Busemann density.
The main tool of Patterson's method is Poincar\'e series.

Given a real number $s$ and a pair of points $p,q$ in $\widetilde{M}$, the $\mathbf{Poincar\acute{e}~series}$ is defined as
	\begin{displaymath}
		P(s,p,q) \triangleq \sum_{\alpha\in\Gamma}\mathrm{e}^{-s d(p,\alpha q)}.
	\end{displaymath}
Then we define the $\mathbf{critical~exponent}$ of this Poincar$\acute{e}$ series as
	\begin{displaymath}
		\delta_\Gamma \triangleq\inf\{s\geq 0\mid P(s,p,q)<\infty\}.
	\end{displaymath}
It's easy to see that the critical~exponent $\delta_\Gamma$ is independent of the choices of the points $p$ and $q$.

\begin{theorem}\label{PS}
Let $\widetilde{M}$ be a simply connected uniform visibility manifolds without conjugate points,
and $\Gamma \subset Iso(\widetilde{M})$ be a discrete group, if the critical exponent $\delta_{\Gamma}<+\infty$,
then there exists at least one $\delta_{\Gamma}-$dimensional Busemann density with supprot exactly equal to the limit set $L(\Gamma)$.
\end{theorem}
\begin{proof}
Fix a point $p_{0}\in \widetilde{M}$, for each $p\in \widetilde{M}$ and $s > \delta_{\Gamma}$, define
$$\mu_{p.p_{0},s} \triangleq \frac{1}{P(s,p_{0},p_{0})}\sum_{\alpha\in \Gamma}e^{-s\cdot d(p,\alpha p_{0})}\mathfrak{D}_{\alpha p_{0}},$$
where $\mathfrak{D}_{\alpha p_{0}}$ is the Dirac measure at the point $\alpha p_{0}$.
By the triangle inequality
$$ -d(p,p_{0})+d(p_{0},\alpha p_{0})  ~\leqslant~ d(p,\alpha p_{0}) ~\leqslant~d(p,p_{0})+d(p_{0},\alpha p_{0}),$$
we know that $$e^{-s\cdot d(p,p_{0})}\leqslant \mu_{p.p_{0},s}(\overline{\widetilde{M}}) \leqslant e^{s\cdot d(p,p_{0})}.$$
Thus for each $p\in \widetilde{M}$ and $s > \delta_{\Gamma}$, $\mu_{p.p_{0},s}$ is a finite measure that satisfies
$\Gamma(p_{0})\subset \mathrm{supp}(\mu_{p.p_{0},s})\subset  \overline{\Gamma(p_{0})}$.
For any sequence $\{s_{k}\}^{+\infty}_{n=1}\subset \mathbb{R}$ satisfies $s_{k}\searrow \delta_{\Gamma}$,
we denote by the weak limit of $\{\mu_{p,p_{0},s_{k}}\}^{+\infty}_{k=1}$ is $\mu_{p}$, i.e., $\mu_{p}=\lim_{k\rightarrow +\infty}\mu_{p,p_{0},s_{k}}$.
One can see that the different choices of $\{s_{k}\}^{+\infty}_{k=1}$ may lead different weak limits. In fact,
we have proved that when the conical limit set has positive measure in $\mu_{p}$, the weak limit is unique up to a a positive multiple (\cite{LLW} Proposition 5.7).

Now we suppose that the Poincar\'e series is of divergent type, i.e., $P(\delta_{\Gamma},p_{0},p_{0})=+\infty$. Since $\Gamma$ is a discrete subgroup,
the support of $\mu_{p}$ will be pushed to the limit set in the ideal boundary, i.e., $\mathrm{supp}(\mu_{p})\subset  \overline{\Gamma(p_{0})} \cap \widetilde{M}(\infty)$.
Thus for each $p\in \widetilde{M}$, $\mu_{p}$ is a positive finite measure on the limit set $L(\Gamma)$.

First, we show that for any $p,q \in\widetilde{M}$ and $\xi\in L(\Gamma)\subset\widetilde{M}(\infty)$, \eqref{eq3.1} is valid.

In fact, suppose that $\alpha_{n}(p_{0})\rightarrow \xi$ in the cone topology, then the ratio of coefficients of $\alpha_{n}(p_{0})$ in
$\mu_{p,p_{0},s_{k}}$ and $\mu_{q,p_{0},s_{k}}$ is
$$\frac{e^{-s_{k}\cdot d(p,\alpha_{n}p_{0})}}{e^{-s_{k}\cdot d(q,\alpha_{n}p_{0})}}=e^{-s_{k}\cdot \{d(p,\alpha_{n}p_{0})-d(q,\alpha_{n}p_{0})\}}.$$
Thus we only need to prove that
\begin{equation}\label{eq3.3}
     \lim_{n\rightarrow +\infty}\{d(p,\alpha_{n}p_{0})-d(q,\alpha_{n}p_{0})\}=\beta_{\xi}(p,q).
\end{equation}

Let $a \triangleq d(p,q)>0$, and $c_{p,q}:[0,a]\rightarrow \widetilde{M}$ be the connecting geodesic segment with
$c_{p,q}(0)=p$ and $c_{p,q}(a)=q$. For any $\epsilon >0$, since $\alpha_{n}(p_{0})\rightarrow \xi$, there exists an $N \in \mathbb{N}$,
such that for any $n \geqslant N$,
\begin{equation}\label{eq3.4}
      \theta_{t}\triangleq \measuredangle_{c_{p,q}(t)}(\xi, \alpha_{n}(p_{0})) < \frac{\epsilon}{a}, ~~~t\in [0,a].
\end{equation}
We view $\{d(p,\alpha_{n}p_{0})-d(\ast,\alpha_{n}p_{0})\}$ and $\beta_{\xi}(p,\ast)$ as functions on $\widetilde{M}$.
By \cite{Ru} Lemma 4.2, we know that
\begin{equation}\label{eq3.5}
      \mathrm{grad}\{d(p,\alpha_{n}p_{0})-d(q,\alpha_{n}p_{0})\}=-c'_{q,\alpha_{n}p_{0}}(0),~~~\mathrm{grad} \beta_{\xi}(p,q)=-c'_{q,\xi}(0),
\end{equation}
where both $c_{q,\alpha_{n}p_{0}}$ and $c_{q,\xi}$ are unit speed connecting geodesics starting from the point $q$.
By \eqref{eq3.4} and \eqref{eq3.5}, for $n \geqslant N$, we have
\begin{displaymath}
		\begin{aligned}
		|\{d(p,\alpha_{n}p_{0})-d(q,\alpha_{n}p_{0})\} - \beta_{\xi}(p,q)|
        & = \left|\int^{a}_{0}\frac{d}{dt}\{ \{d(p,\alpha_{n}p_{0})-d(c_{p,q}(t),\alpha_{n}p_{0})\}- \beta_{\xi}(p,c_{p,q}(t)) \} dt  \right|    \\
	    & =   \left|\int^{a}_{0}<-c'_{c_{p,q}(t),\alpha_{n}p_{0}}(0)+c'_{c_{p,q}(t),\xi}(0),c'_{p,q}(t)> dt  \right|                              \\
        & \leqslant   \int^{a}_{0} \|-c'_{c_{p,q}(t),\alpha_{n}p_{0}}(0)+c'_{c_{p,q}(t),\xi}(0) \|   \cdot \|c'_{p,q}(t)\|dt   \\
        & \leqslant   \int^{a}_{0}\theta_{t}dt ~ < ~  \epsilon.  \\
		\end{aligned}
\end{displaymath}
Thus \eqref{eq3.3} holds, and and consequently, \eqref{eq3.1} also holds.

Next, we show that \eqref{eq3.2} holds.

In fact, for any $\alpha\in \Gamma$ and $s_{k}>\delta_{\Gamma}$,
\begin{displaymath}
		\begin{aligned}
		\mu_{\alpha p,p_{0},s_{k}}
        & = \frac{1}{P(s_{k},p_{0},p_{0})}\sum_{\gamma\in \Gamma}e^{-s_{k}\cdot d(\alpha p,\gamma p_{0})}\mathfrak{D}_{\gamma p_{0}}    \\
	    & =  \frac{1}{P(s_{k},p_{0},p_{0})}\sum_{\gamma\in \Gamma}e^{-s_{k}\cdot d(\alpha p,\alpha\gamma p_{0})}\mathfrak{D}_{\alpha\gamma p_{0}}                             \\
        & =  \frac{1}{P(s_{k},p_{0},p_{0})}\sum_{\gamma\in \Gamma}e^{-s_{k}\cdot d(p,\gamma p_{0})}\mathfrak{D}_{\alpha\gamma p_{0}}.      \\
		\end{aligned}
\end{displaymath}
Therefore for any Borel subset $B\subset \overline{\widetilde{M}}$,
$\mu_{\alpha p,p_{0},s_{k}}(\alpha B)=\mu_{p,p_{0},s_{k}}(B)$. Let $k\rightarrow +\infty$, we have that for any
measurable subset $A\subset\widetilde{M}(\infty)$, $\mu_{\alpha p}(\alpha A)=\mu_p(A)$, i.e., \eqref{eq3.2} holds.

Last, we only need to prove that for each $p\in \widetilde{M}$, $\mathrm{supp}(\mu_{p})=L(\Gamma)$.
This is an easy consequence by \eqref{eq3.1} and Theorem~\ref{mini}.

Note that all the discussions above are under the assumption that $\Gamma$ is of divergent type,
when $\Gamma$ is not of divergent type, Patterson proposed a clever method to overcome this difficulty in \cite{Pa}.
He constructed a positive monotonic increasing function $h$ defined on $\mathbb{R}^{+}$, such that the modified Poincar\'e series
$$\widetilde{P}(s,p,q)=\sum_{\alpha\in\Gamma}h(d(p,\alpha q))\mathrm{e}^{-s\cdot d(p,\alpha q)}$$
has the same critical exponent with the original Poincar\'e series $P(s,p,q)$, and $\widetilde{P}(s,p,q)$ is of divergent type.
Then one can easily check that both \eqref{eq3.1} and \eqref{eq3.2} hold for the modified Poincar\'e series $\widetilde{P}(s,p,q)$.
\end{proof}

A Busemann density constructed as the way described in Theorem~\ref{PS} is called a $\delta_{\Gamma}$-dimensional $\mathbf{Patterson}$-$\mathbf{Sullivan~measure}$.
In the following text, when we say ``Patterson-Sullivan measure", what we mean is always the $\delta_{\Gamma}$-dimensional Busemann density.

The following Mohsen shadow lemma was proved in our previous paper~\cite{LLW}.
For any point $p \in \widetilde{M}$ and any subset $A \subset \widetilde{M}$,
the $\mathbf{shadow~of~A~in~the~ideal~boundary~from~p}$ is defined as
$$\mathrm{pr}_{p}(A) \triangleq \{c_{p,z}(+\infty)| z \in A\}\subset \widetilde{M}(\infty).$$

\begin{proposition}[Mohsen Shadow Lemma~\cite{LLW}]\label{shadow}
	Let $\widetilde{M}$ be a complete simply connected and visibility manifold without conjugate points and $\Gamma$ is a non-elementary discrete subgroup of $\text{Iso}(\widetilde{M})$. Suppose ${\{\mu_p\}}_{p\in\widetilde{M}}$
is an $r$-dimensional Busemann density and $K\subset\widetilde{M}$ is a compact set,
then for $R>0$ large enough, there exists a constant $C>0$, such that for any $\alpha\in\Gamma$, and any pair of points $p,q\in K$, we have
	\begin{displaymath}
		\frac 1C\leq \frac{\mu_p(\mathrm{pr}_p(B(\alpha q,R)))}{\mathrm{e}^{-r\cdot d(p,\alpha q)}}\leq C.
	\end{displaymath}
\end{proposition}

\begin{proposition}\label{le1}
Let $\widetilde{M}$ be a complete simply connected and visibility manifold without conjugate points and $\Gamma$ is a non-elementary discrete subgroup of $\text{Iso}(\widetilde{M})$.
Let $r\in \mathbb{R}$ and $\{\mu_{p}\}_{p\in \widetilde{M}}$ be an $r$-dimensional Busemann density, then
\begin{enumerate}
		\item For any $p,q\in\widetilde{M}$, there exists a positive constant $D=D(p,q)$, such that for any $n \in \mathbb{N}$,
           $$\sum_{\alpha \in \Gamma,~n-1<d(p,\alpha q)\leqslant n} e^{-r\cdot d(p,\alpha q)} \leqslant D. $$
		\item $r \geqslant \delta_{\Gamma}.$
	\end{enumerate}
\end{proposition}
\begin{proof}
$1$. ~Denote by $K\triangleq \{p,q\}$. let $R>0$ and $C=C(x,y)>0$ be the constants in the Mohsen shadow lemma (Proposition~\ref{shadow}).
Let $\mathfrak{a} \triangleq \# \{\alpha \in \Gamma ~|~d(q,\alpha q) \leqslant 1+4R\}$ and
$\Gamma_{n}\triangleq \{\alpha \in \Gamma ~|~n-1<d(p,\alpha q) \leqslant n\}$.
We can see that $\mathfrak{a}<\infty$ since $\Gamma$ is a discrete subgroup.

For any $\alpha \in \Gamma_{n}$ and $\xi \in \mathrm{pr}_{p}(B(\alpha q, R))$,
take any point in the intersections of the geodesic ray $c_{p,\xi}$ and the open ball $B(\alpha q, R)$, and denote it by $q_{\alpha}$,
we have that $$n-1-R < d(p,q_{\alpha})\leqslant n + R.$$ Thus if $\xi\in \mathrm{pr}_{p}(B(\alpha q, R)) \cap \mathrm{pr}_{p}(B(\beta q, R))$,
where $\alpha, \beta\in \Gamma_{n}$, we have
\begin{displaymath}
		\begin{aligned}
		d(\alpha q, \beta q)
        & \leqslant d(\alpha q, q_{\alpha}) + d(q_{\alpha},q_{\beta}) + d(q_{\beta},\beta q)   \\
	    & =  d(\alpha q, q_{\alpha}) + | d(q,q_{\beta})-d(q,q_{\alpha}) | + d(q_{\beta},\beta q)                \\
        & \leqslant R + ((n+R)-(n-1-R))+R        \\
        & = 1+ 4R.        \\
		\end{aligned}
\end{displaymath}
Hence for any $\xi \in \widetilde{M}(\infty)$, it can be shadowed by at most $\mathfrak{a}$ $B(\alpha q,R)$ from the point $p$, where $\alpha\in \Gamma_{n}$.
Thus
$$\sum_{\alpha\in \Gamma_{n}}\mu_{p}(\mathrm{pr}_{p}(B(\alpha q, R))) \leqslant
\mathfrak{a} \mu_{p}\left(\bigcup_{\alpha\in \Gamma_{n}}\mathrm{pr}_{p}(B(\alpha q, R))\right). $$
Furthermore, by Mohsen shadow lemma, we have
\begin{displaymath}
		\begin{aligned}
		\sum_{\alpha\in \Gamma_{n}}e^{-r\cdot d(p,\alpha q)}
        & \leqslant C\sum_{\alpha\in \Gamma_{n}}\mu_{p}(\mathrm{pr}_{p}(B(\alpha q, R))) \\
	    &  \leqslant C\mathfrak{a} \mu_{p}\left(\bigcup_{\alpha\in \Gamma_{n}}\mathrm{pr}_{p}(B(\alpha q, R))\right)\\
        &  \leqslant C\mathfrak{a} \|\mu_{p}\|,\\
		\end{aligned}
\end{displaymath}
where $ \|\mu_{p}\|\triangleq \mu_{p}(\widetilde{M}(\infty))<\infty$.
Denote by $D \triangleq C\mathfrak{a} \|\mu_{p}\|<\infty$, then the first assertion of this lemma holds.

$2$. ~Denote by $a_{n}\triangleq \#\Gamma_{n}=\#\{\alpha \in \Gamma ~|~n-1<d(p,\alpha q) \leqslant n\}$.

$\mathbf{Case ~~I}$. ~~~$r \geqslant 0$.

Due to the first assertion of this lemma, we have that
$$a_{n}e^{-rn}
\leqslant \sum_{\alpha\in \Gamma_{n}}e^{-r\cdot d(p,\alpha q)}
\leqslant C\mathfrak{a} \|\mu_{p}\| = D.$$
Thus $\frac{1}{n}\ln(a_{n}e^{-rn}) \leqslant \frac{1}{n}\ln D$, let $n\rightarrow +\infty$, we get $r \geqslant \delta_{\Gamma}$.

$\mathbf{Case ~~II}$. ~~~$r < 0$.

Similar to the Case I, we have
$$a_{n}e^{-r(n-1)}
\leqslant \sum_{\alpha\in \Gamma_{n}}e^{-r\cdot d(p,\alpha q)}
\leqslant C\mathfrak{a} \|\mu_{p}\| = D.$$
Thus
$\frac{1}{n}\ln(a_{n}e^{-r(n-1)}) \leqslant \frac{1}{n}\ln D$, let $n\rightarrow +\infty$, we get $r \geqslant \delta_{\Gamma}$.
\end{proof}

\begin{theorem}[Myrberg Type Dichotomy]\label{Myr}
Let $M$ be a complete uniform visibility manifold without conjugate points that satisfies Axiom 2,
${\{\mu_q\}}_{q\in\widetilde{M}}$ be a Patterson-Sullivan measure, $p\in\widetilde{M}$ is arbitrarily chosen and $\mathfrak{m}$ is the
corresponding $\delta_{\Gamma}$-dimensional BMS measure. Then the following assertions are equivalent:
\begin{enumerate}
		\item[1.] The Myrberg limit set $L_{m}(\Gamma)$ satisfies that $\mu_{p}(L_{m}(\Gamma)) >0$.
		\item[2.] The geodesic flow $\phi_t: T^{1}M \rightarrow T^{1}M $ is conservative with respect to $\mathfrak{m}$.
	\end{enumerate}
\end{theorem}
\begin{proof}
$``\Longrightarrow"$~~Suppose that $\mu_{p}(L_{m}(\Gamma)) >0$, by Proposition~\ref{myr},
$\mu_{p}(L_{c}(\Gamma)) \geqslant \mu_{p}(L_{m}(\Gamma))>0$.
Then the Hopf-Tsuji-Sullivan dichotomy (Theorem~\ref{HTS}) implies that $\mu_{p}(L_{c}(\Gamma)) =\mu_{p}(\widetilde{M}(\infty))$,
thus again by the HTS dichotomy, the geodesic flow $\phi_t$ is conservative with respect to the BMS measure $\mathfrak{m}$.

$``\Longleftarrow"$~~Let $\mathfrak{S}$ be the countable basis of the $\mathrm{supp} (\widetilde{\mathfrak{m}})$,
for each $\emptyset \neq A \in \mathfrak{S}$, denote by
$$\mathfrak{L}(A)\triangleq \left\{v \in T^{1}\widetilde{M} ~\mid ~\exists \{\alpha_{n}\}^{\infty}_{n=1}\subset \Gamma,
\{t_{n}\}^{\infty}_{n=1}\subset \mathbb{R}, s.t. ~t_{n}\rightarrow +\infty ~and ~\phi_{t_{n}}(v)\in \alpha_{n}(A)  \right\}.$$
It's easy to see that $\mathfrak{L}(A)$ is an invariant set of the geodesic flow.
Since the geodesic flow $\phi_t: T^{1}M \rightarrow T^{1}M $ is conservative with respect to $\mathfrak{m}$,
there exists a subset $A'\subset A$ such that $\mathfrak{\widetilde{m}}(A')=\mathfrak{\widetilde{m}}(A)$ and $A'\subset \mathfrak{L}(A)$,
thus $\mathfrak{\widetilde{m}}(\mathfrak{L}(A))\geqslant \mathfrak{\widetilde{m}}(A')=\mathfrak{\widetilde{m}}(A) >0$.
By the HTS dichotomy (Theorem~\ref{HTS}), the geodesic flow $\phi_t: T^{1}M \rightarrow T^{1}M $ is ergodic with respect to $\mathfrak{m}$,
while $\mathfrak{L}(A)$ is an invariant set of the geodesic flow on the universal covering manifold with positive measure,
we know that $\mathfrak{L}(A)$ has full $\mathfrak{\widetilde{m}}$-measure.
By the fact that
$$L_{m}(\Gamma)=\left\{c_{v}(+\infty)~|~v\in \bigcap_{A\in \mathfrak{S}} \mathfrak{L}(A)\right\}$$
and the quai-product structure of $\mathfrak{\widetilde{m}}$, the preimage of a positive $\mu_{p}$-measure set under the map
$v\longmapsto c_{v}(+\infty)$ is a positive $\mathfrak{\widetilde{m}}$-measure set, thus $\mu_{p}(L_{m}(\Gamma))=\mu_{p}(L(\Gamma))>0$.
\end{proof}

In his seminal work \cite{My}, Myrberg showed that for Fuchsian groups, the Myrberg limit set $L_{m}(\Gamma)$ has full linear measure in $\mathbb{S}^{1}$
 (in this case, the ideal boundary is homeomorphic to the unit circle). Then Agard in \cite{Aq} generalized this result and prove that,
 for an n-dimensional hyperbolic manifold, the Hausdorff measure gives full measure for $L_{m}(\Gamma)$ in the ideal boundary.
In~\cite{Tu}, Tukia showed that for an n-dimensional hyperbolic manifold,
the Myrberg limit set $L_{m}(\Gamma)$ is a full $\mu_{p}$-measure subset of the conical limit set $L_{c}(\Gamma)$.
An immediately consequence of Theorem~\ref{Myr} and Theorem~\ref{HTS} is the following result.

\begin{corollary}\label{cor}
If the geodesic flow $\phi_t: T^{1}M \rightarrow T^{1}M $ is conservative with respect to the BMS measure $\mathfrak{m}$,
then $\mu_{p}(L_{m}(\Gamma))=\mu_{p}(L_{c}(\Gamma))$.
Thus the first item of Theorem~\ref{Myr} can be replaced by ``The Myrberg limit set $L_{m}(\Gamma)$ has full $\mu_{p}$-measure".
\end{corollary}

A unit vector $v \in T^{1}M$ is called a $\mathbf{non}$-$\mathbf{wandering~point}$ for the geodesic flow
$\phi_{t}: T^{1}M \rightarrow T^{1}M$, if there exists a sequence of unit vectors $\{v_{n}\}^{^{\infty}}_{n=1}\subset T^{1}M $
and a sequence of positive real numbers $\{t_{n}\}^{^{\infty}}_{n=1}\subset \mathbb{R}$, such that
$$v_{n}\rightarrow v, ~~~t_{n}\rightarrow +\infty,~~~\phi_{t_{n}}(v_{n})\rightarrow v.$$
The set $$\mathbf{\Omega(\Gamma)}\triangleq \{v\in T^{1}M ~|~v~is~a~non-wandering~point\}$$
is called the $\mathbf{non}$-$\mathbf{wandering~set}$ of the geodesic flow,
and we denote by $\mathbf{\widetilde{\Omega}(\Gamma)}\subset T^{1}\widetilde{M}$ the lift of $\Omega(\Gamma)$ in $T^{1}\widetilde{M}$.

\begin{theorem}\label{unique}
Let $M$ be a complete uniform visibility manifold without conjugate points that satisfies Axiom 2,
and ${\{\mu_q\}}_{q\in\widetilde{M}}$ be be an $r$-dimensional Busemann density,
$p\in\widetilde{M}$ is arbitrarily chosen and $\mathfrak{m}$  is the $r$-dimensional BMS measure defined by $\mu_{p}$.
If $\mathfrak{m}$ is a finite measure on $T^{1}M$, then $r=\delta_{\Gamma}$, and $\Gamma$ is of divergent type,
and there exists a unique (up to s positive scalar multiple) BMS measure, which supported exactly on the non-wandering set $\Omega(\Gamma)$,
and the geodesic flow on $T^{1}M$ is conservative and ergodic with respect to the BMS measure $\mathfrak{m}$ .
\end{theorem}

In order to prove this Theorem, we need the following result,
which also shows that why the limit set is so important in the dynamics of the geodesic flows.

\begin{theorem}\label{nonwandering}
Let $\widetilde{M}$ be a simply connected uniform visibility manifold without conjugate points,
$\Gamma \subset Iso(\widetilde{M})$ be a discrete subgroup and denote by $M \triangleq \Gamma\backslash \widetilde{M}$,
then $$\widetilde{\Omega}(\Gamma)=\{\widetilde{v} \in T^{1}\widetilde{M}~|~c_{\widetilde{v}}(-\infty)\in L(\Gamma),~c_{\widetilde{v}}(+\infty)\in L(\Gamma)\}.$$
\end{theorem}
\begin{proof}
First, we show that the relation ``$\subseteq$" holds.

Suppose $v\in \Omega(\Gamma) \subseteq T^{1}M$, by definition, there exists a sequence of unit vectors $\{v_{n}\}^{^{\infty}}_{n=1}\subset T^{1}M $
and a sequence of numbers $\{t_{n}\}^{^{\infty}}_{n=1}\subset \mathbb{R}$, such that
\begin{equation}\label{eq2.8}
		v_{n}\rightarrow v, ~~~t_{n}\rightarrow +\infty,~~~\phi_{t_{n}}(v_{n})\rightarrow v.
\end{equation}
Let $\widetilde{v}\in T^{1}\widetilde{M}$ be a lift of $v$, then by~\eqref{eq2.8}, for each $n\in \mathbb{N}^{+}$,
there exists $\widetilde{v}_{n}\in T^{1}\widetilde{M}$, a lift of $v_{n}$ in $T^{1}\widetilde{M}$, and
$\{\alpha_{n}\}_{n=1}^{\infty}\subset \Gamma$, such that
\begin{equation}\label{eq2.9}
		\widetilde{v}_{n}\rightarrow \widetilde{v},~~~\alpha_{n}\circ c'_{\widetilde{v}_{n}}(t_{n})\rightarrow \widetilde{v}.
\end{equation}
Denote by $q \triangleq c_{\widetilde{v}}(0)$ and $q_{n} \triangleq c_{\widetilde{v}_{n}}(0)$ for each $n\in \mathbb{N}^{+}$,
thus $$d(c_{\widetilde{v}_{n}}(t_{n}),\alpha^{-1}_{n}(q))=d(\alpha_{n}c_{\widetilde{v}_{n}}(t_{n}),q)\rightarrow 0.$$
Thus by Proposition~\ref{prop_2_7}(1), we have that $\alpha^{-1}_{n}(q)\rightarrow c_{\widetilde{v}}(+\infty)$ in the cone topology,
so $c_{\widetilde{v}}(+\infty)\in L(\Gamma)$.

Now we only need to show that $c_{\widetilde{v}}(-\infty)\in L(\Gamma)$. For this, we denote by
$$\gamma_{n}(t) \triangleq \alpha_{n}\circ c_{\widetilde{v}_{n}}(t_{n}-t),~~~t\in \mathbb{R}.$$
Then by~\eqref{eq2.9}, the tangent vector of geodesic $\gamma_{n}$ at $t=0$ satisfies
$$\gamma'_{n}(0) = -\alpha_{n}\circ c'_{\widetilde{v}_{n}}(t_{n})\rightarrow -\widetilde{v}.$$
Furthermore by Proposition~\ref{prop_2_7}(1), we have
$$\alpha_{n}(q_{n})=\alpha_{n}\circ c'_{\widetilde{v}_{n}}(t_{n}-t_{n})=\gamma_{n}(t_{n})\rightarrow
c_{-\widetilde{v}}(+\infty) = c_{\widetilde{v}}(-\infty).$$
Since $q_{n}\rightarrow q$, the uniform visibility axiom implies that
$$\alpha_{n}(q)\rightarrow c_{\widetilde{v}}(-\infty)~\Rightarrow~c_{\widetilde{v}}(-\infty)\in L(\Gamma).$$
Therefore we have proved the ``$\subseteq$" part.

Next, let's prove the ``$\supseteq$" part.

Suppose that $\widetilde{v} \in T^{1}\widetilde{M}$ satisfies that $c_{\widetilde{v}}(-\infty)\in L(\Gamma),~c_{\widetilde{v}}(+\infty)\in L(\Gamma)$.
By Proposition~\ref{prop_2_7}(3), we know that $c_{\widetilde{v}}(-\infty)$ and $c_{\widetilde{v}}(+\infty)$ are $\Gamma$-dual,
thus there exist $\{\alpha_{n}\}^{+\infty}_{n=1}\subseteq \Gamma$, such that
$$\alpha^{-1}_{n}(p)\rightarrow c_{\widetilde{v}}(-\infty),~~~ \alpha_{n}(p)\rightarrow c_{\widetilde{v}}(+\infty)~~~\forall p\in \widetilde{M}.$$
Denote by $q \triangleq c_{\widetilde{v}}(0)$, then by the definition of cone topology, we have
$$t_{n} \triangleq d(q,\alpha^{-1}_{n}(q))\rightarrow +\infty.$$
Let $\widetilde{v}_{n}\triangleq c'_{\alpha_{n}^{-1}(q),q}(0)$, so $\phi_{t_{n}}(\widetilde{v}_{n})=-c'_{q,\alpha_{n}^{-1}(q)}(0)$.
Since $\alpha^{-1}_{n}(q)\rightarrow c_{\widetilde{v}}(-\infty)$, without loss of generality, we can assume that
$$\measuredangle_{q}(\alpha^{-1}_{n}(q),c_{\widetilde{v}}(-\infty)) < \frac{1}{n},~~~n\in \mathbb{N}^{+}.$$
Thus
$$\measuredangle_{q}(\phi_{t_{n}}(\widetilde{v}_{n}),\widetilde{v})=
\measuredangle_{q}(\alpha^{-1}_{n}(q),c_{\widetilde{v}}(-\infty)) < \frac{1}{n},~~~n\in \mathbb{N}^{+},$$
therefore
\begin{equation}\label{eq2.10}
		\phi_{t_{n}}(\widetilde{v}_{n})\rightarrow \widetilde{v}.
\end{equation}
Since $\alpha_{n}(q)\rightarrow c_{\widetilde{v}}(+\infty)$, we know that
\begin{equation}\label{eq2.11}
		\alpha_{n}(\widetilde{v}_{n})=c'_{q,\alpha_{n}(q)}(0)\rightarrow \widetilde{v}.
\end{equation}
Let $\mathcal{C}:\widetilde{M}\rightarrow M$ be the covering map, then by~\eqref{eq2.11}, we have
$$d\mathcal{C}(\widetilde{v}_{n})=d\mathcal{C}(\alpha_{n}\circ \widetilde{v}_{n}) \rightarrow d\mathcal{C}(\widetilde{v}).$$
Denote by $v\triangleq d\mathcal{C}(\widetilde{v})$ and $v_{n}\triangleq d\mathcal{C}(\widetilde{v}_{n})$, then by~\eqref{eq2.10}
$$\phi_{t_{n}}(v_{n})=d\mathcal{C}(\phi_{t_{n}}\widetilde{v}_{n})\rightarrow d\mathcal{C}(\widetilde{v})=v,$$
thus $v\in \Omega(\Gamma)$, therefore $\widetilde{v}\in \widetilde{\Omega}(\Gamma)$. We proved the ``$\supseteq$" part.
\end{proof}

\noindent\emph{Proof of Theorem \ref{unique}}.
Since $\mathfrak{m}$ is a finite $\phi_{t}$-invariant measure, by the Poincar\'e recurrence theorem,
the geodesic flow $\phi_t$ is conservative with respect to the BMS measure $\mathfrak{m}$.
Then the HTS dichotomy (Theorem~\ref{HTS}) implies that the Poincar$\acute{e}$ series is divergent at $r$,
while Proposition~\ref{le1} shows that $r \geqslant \delta_{\Gamma}$, thus $r=\delta_{\Gamma}$, and $\Gamma$ is of divergent type.

By the HTS dichotomy, we know that $\phi_t$ is ergodic with respect to $\mathfrak{m}$, and the conical limit set $L_{c}(\Gamma)$
has full $\mu_{p}$-measure, thus by the Proposition 5.7 in~\cite{LLW}, the Patterson-Sullivan measure is unique up to s positive scalar multiple,
thus the BMS measure $\mathfrak{m}$ corresponding to the Patterson-Sullivan measure is unique.
$\mathrm{supp} (\mathfrak{m})=\Omega(\Gamma)$ follows easily by Theorem~\ref{mini} and Theorem~\ref{nonwandering}.

\section{\bf Some Related Questions}\label{question}
\setcounter{equation}{0}\setcounter{theorem}{0}

As mentioned in the introduction, on the one hand, the absence of conjugate points provides (almost) no useful information on the local geometry of the manifold;
on the other hand, in many situations, we need to do a series of estimations to get conclusions.
For this, we add the condition of uniform visibility.

It's well known that the Myrberg limit points are conical limit points for negatively curved manifolds.
In Proposition~\ref{myr}, we have showed that this is also true for the manifolds without conjugate points, under the condition of uniform visibility.

\begin{question}\label{q1}
Let $M=\Gamma \backslash\widetilde{M}$ be a complete Riemannian manifolds without conjugate points, do we have $L_{m}(\Gamma) \subset L_{c}(\Gamma)$?
\end{question}

Theorem~\ref{nonwandering} gives the relationship between the limit set and the non-wandering set of the geodesic flow.

\begin{question}\label{q2}
Does Theorem~\ref{nonwandering} remains true if we remove the hypothesis ``uniform visibility"?
\end{question}

For a long time, the ergodicity of geodesic flows on the closed surfaces of non-positive curvature (and surfaces without focal points, surfaces without conjugate points)
 with respect to the Liouville measures is a major open problem. Together with Weisheng Wu and Fang Wang (cf. \cite{Wu1},~\cite{WLW}),
we have solved this problem for surfaces without focal points of genus greater than $1$,
under the assumption that the set of points of the surface with negative curvature has at most finitely many connected components.
Recently, Weisheng Wu (cf. \cite{Wu1}) has made significant progress on surfaces without conjugate points under the assumption ``bounded asymptote".
We note that such surfaces are always uniform visibility (cf. \cite{Eb1}).
One of the difficulties to this problem is caused by flat stripes,
thus we have the following question.

\begin{question}\label{q3}
For a closed surface without conjugate points with genus greater than $1$,
if we further assume that the surface satisfies the Axiom $2$, is the geodesic flow ergodic with respect to the Liouville measure?
\end{question}

Although in this paper, the manifold $M$ we considered is not necessarily compact, we ask the following question about compact manifold without conjugate points.

\begin{question}\label{q4}
For a compact manifold without conjugate points and dimension greater than $2$,  are the visibility and uniform visibility axioms equivalent?
\end{question}

It is known that for the compact manifolds of non-positive curvature and closed surface without conjugate points,
the answers to the Question~\ref{q4} are ``Yes".

\section*{\textbf{Acknowledgements}}
The author is partially supported by Natural Science Foundation of Shandong Province under Grant No.~ZR2020MA017 and No.~ZR2021MA064.

\end{document}